\documentclass{amsart}

\usepackage{amsmath}
\usepackage{amsthm}
\usepackage{amsfonts}
\usepackage{amssymb}

\usepackage{pstricks-add}

\theoremstyle{plain}

\newtheorem{theorem}{Theorem}

\newtheorem{lemma}[theorem]{Lemma}

\theoremstyle{definition}

\include{psfig}

\begin{document}
\title{Finding the Fermat Point via analysis}
\author{Luqing Ye}
\address{An undergrad at College of Science, Hangzhou Normal University,Hangzhou City,Zhejiang Province,China}
\email{yeluqingmathematics@gmail.com}

\maketitle

 Let $P_1,P_2,P_3$ be three given points in $\mathbf{R}^2$ ,and $P$ be an arbitrary
 point in $\mathbf{R}^2$.The classical Fermat's problem to
 Torricelli asks for the location of $P$,such that 
 \begin{equation*}
   |PP_1|+|PP_2|+|PP_3|
 \end{equation*}
is a minimum.Then $P$ is called the Fermat point of the triangle
$P_1P_2P_3$(Triangle $P_1P_2P_{3}$ is  nondegenerate).There exist several elegant geometrical solutions in the
literature.In this note,we consider finding the Fermat point by
using methods in advanced calculus.The main tools we use are the 
extreme value theorem,Fermat's theorem,and the intermediate value
theorem,which are listed below.
\begin{theorem}[The extreme value theorem]
Let $f:D\to \mathbf{R}$ be a continuous function,where $D$ is a
nonempty bounded closed set in $\mathbf{R}^2$.Then $f$ must attain a minimum
on $D$.That is,there exists a point $\xi$ in $D$ such that $f(\xi)\leq
f(x)$ for all $x\in D$.
\end{theorem}
\begin{theorem}[Fermat's theorem]
  Let $f:C\to \mathbf{R}$ be a differentiable function,where $C$ is a
  nonempty open set in $\mathbf{R}^2$.Suppose $x_0\in C$ is a local
  extreme point of $f$,then $f'(x_0)$ is a zero linear map\footnote{A zero
    linear map maps any vectors to the zero vector.} from
  $\mathbf{R}^2$ to $\mathbf{R}$.
\end{theorem}
\begin{theorem}[The intermediate value theorem]
  Let $f:I\to \mathbf{R}$ be a continuous function,where $I=[a,b]$ is a
  closed interval of $\mathbf{R}$.For any real number $u$ between
  $f(a)$ and $f(b)$,there exists a $\xi\in [a,b]$ such that $f(\xi)=u$.
\end{theorem}

And we need four lemmas.
\begin{lemma}\label{lemma:4}
  Let $M_1M_2M_3$ be a nondegenerate triangle,let $M$ be an arbitrary
  point in the interior of the triangle region.As shown in figure
  \eqref{fig:0}. Then $$\hbox{rad}\angle M_1MM_3>\hbox{rad}\angle
  M_1M_2M_3$$,where $\hbox{rad}\angle M_1MM_3$ is the radian measure of
  the angle $\angle M_1MM_3$.
\begin{proof}
  Extend the segment $M_3M$ to $N$,where $N$ is a point on the segment $M_1M_2$.Then  $\hbox{rad}\angle M_1MM_3>\hbox{rad}\angle M_1NM_3$,and $\hbox{rad}\angle M_1NM_3>\hbox{rad}\angle M_1M_2M_3$.So $\hbox{rad}\angle M_1MM_3>\hbox{rad}\angle M_1M_2M_3$.
\end{proof}
\begin{figure}[h]
\psset{xunit=1.0cm,yunit=1.0cm,algebraic=true,dotstyle=o,dotsize=3pt 0,linewidth=0.8pt,arrowsize=3pt 2,arrowinset=0.25}
\begin{pspicture*}(1.3,-5.88)(23.02,6.3)
\psline(6.3,3.38)(3.42,-1.66)
\psline(3.42,-1.66)(12.08,-1.62)
\psline(12.08,-1.62)(6.3,3.38)
\psline(6.3,3.38)(7.08,0.46)
\psline(7.08,0.46)(3.42,-1.66)
\psline[linestyle=dashed,dash=5pt 5pt](8.63,1.36)(7.08,0.46)
\begin{scriptsize}
\psdots[dotstyle=*](6.3,3.38)
\rput[bl](6.06,3.7){{$M_1$}}
\psdots[dotstyle=*](3.42,-1.66)
\rput[bl](2.9,-1.62){{$M_3$}}
\psdots[dotstyle=*](12.08,-1.62)
\rput[bl](12.5,-1.68){{$M_2$}}
\psdots[dotstyle=*](7.08,0.46)
\rput[bl](7.16,0.58){{$M$}}
\psdots[dotstyle=*](8.63,1.36)
\rput[bl](8.6,1.6){{$N$}}
\end{scriptsize}
\end{pspicture*}
  \caption{}
  \label{fig:0}
\end{figure}
\end{lemma}

\begin{lemma}
  For unit vectors $\mathbf{e_1,e_2,e_3}$ on the Euclidean plane $\mathbf{R}^2$,
$$
\mathbf{e_1+e_2+e_3=0}
$$
if and only if the radian  measure of the angle  between any two of the unit
vectors is $\frac{2\pi}{3}$.
\end{lemma}
\begin{proof}
  The proof is easy via plane geometry,so is left to the reader.
\end{proof}
\begin{lemma}\label{theorem:1.22}
When the radian measure of all the interior angles of the triangle
$P_1P_2P_3$ are less than $\frac{2\pi}{3}$,then we can
find a unique point $F'$ in the interior of the triangle region
satisfying
$$
\mbox{rad}\angle P_1F'P_2=\mbox{rad}\angle P_2F'P_3=\mbox{rad}\angle P_3F'P_1=\frac{2\pi}{3}.
$$
\end{lemma}

\begin{proof}
As shown in figure
\eqref{fig:1},suppose that an interior  point  of
the triangle region $F$ is on a circle which passes through
$P_2,P_3$,then $\hbox{rad}\angle P_3FP_2$ is a
constant,i.e,$\hbox{rad}\angle P_3FP_2$ remains unchanged when $F$
moves on the circle.Now let this circle move while keeping the property that the circle passes through
$P_2,P_3$.When the center of this circle moves downward to
infinity,$\hbox{rad}\angle P_3FP_2$ tends to $\pi$.When
the center of this circle moves from infinity to a location such that the circle passes through
$P_1,P_2$ and $P_3$,then $\hbox{rad}\angle P_3FP_2$ becomes
$\hbox{rad}\angle P_3P_1P_2$,which is less than $\frac{2\pi}{3}$.So according to
the intermediate value theorem,there exists a location $G$ such that when the center of this circle moves to
$G$,then $\hbox{rad}\angle P_3FP_2$ becomes
$\frac{2\pi}{3}$.Denote the circle centering at $G$ by $O'$.Now let $F$ move on $O'$.When $F$ tends to the line $P_1P_{3}$,$\hbox{rad}\angle P_1FP_3$ tends to
$\pi$ while $\hbox{rad}\angle P_1FP_2$ tends to $2\pi-\pi-\frac{2\pi}{3}=\frac{\pi}{3}$.
When $F$ tends to the line $P_1P_2$,
$\hbox{rad}\angle P_1FP_2$ tends to $\pi$ while $\hbox{rad}\angle P_1FP_3$ tends
to $2\pi-\pi-\frac{2\pi}{3}=\frac{\pi}{3}$.So according to the
intermediate value theorem,there exists a point $F'$ on the circle $O'$ such
that $\hbox{rad}\angle P_1F'P_3=\hbox{rad}\angle P_1F'P_{2}$,i.e,both of them are equal to
$\frac{2\pi-\frac{2\pi}{3}}{2}=\frac{2\pi}{3}$.So
$$
\mbox{rad}\angle P_1F'P_2=\mbox{rad}\angle P_2F'P_3=\mbox{rad}\angle P_3F'P_1=\frac{2\pi}{3}.
$$
And the uniqueness of the point $F'$ is obvious by lemma 4.
\end{proof}
\begin{figure}[h]
\psset{xunit=1.0cm,yunit=1.0cm,algebraic=true,dotstyle=o,dotsize=3pt 0,linewidth=0.8pt,arrowsize=3pt 2,arrowinset=0.25}
\begin{pspicture*}(2,-6.67)(21.62,4.3)
\psline(7.66,3.72)(5.22,-0.84)
\psline(5.22,-0.84)(12.22,-0.84)
\psline(12.22,-0.84)(7.66,3.72)
\psline(8.72,-6.67)(8.72,4.3)
\pscircle(8.72,0.21){3.65}
\pscircle(8.72,-0.45){3.52}
\pscircle(8.72,-2.56){3.9}
\psline(7.66,3.72)(7.93,1.26)
\psline(7.93,1.26)(5.22,-0.84)
\psline(7.93,1.26)(12.22,-0.84)
\begin{scriptsize}
\psdots[dotstyle=*](7.66,3.72)
\rput[bl](7.73,3.83){{$P_1$}}
\psdots[dotstyle=*](5.22,-0.84)
\rput[bl](4.44,-0.94){{$P_3$}}
\psdots[dotstyle=*](12.22,-0.84)
\rput[bl](12.29,-0.73){{$P_2$}}
\psdots[dotstyle=*](8.72,0.21)
\psdots[dotstyle=*](8.72,-0.45)
\rput[bl](8.79,-0.35){}
\psdots[dotstyle=*](8.72,-2.56)
\rput[bl](8.79,-2.45){$G$}
\psdots[dotstyle=*](7.93,1.26)
\rput[bl](8,1.36){{$F$}}
\end{scriptsize}
\end{pspicture*}
\caption{}\label{fig:1}
\end{figure}

\begin{lemma}
  Let $f:\mathbf{R}^2\to \mathbf{R}$ be continuous,and $f$ be differentiable
  in the deleted neighborhood\footnote{The deleted neighborhood of a
    point is the neighborhood of the point excluding the point
    itself.} of $P\in \mathbf{R}^2$.A line $l$ in
  $\mathbf{R}^2$ passes through $P$.For any sequence of points $Q_1,Q_2,\cdots,Q_n,\cdots$ on
  $l$ such that $\lim_{n\to\infty}|Q_nP|=0$($P$ is not in the sequence),if
  $\lim_{n\to\infty} f'(Q_n)$ exists and is not a zero linear map from
  $\mathbf{R}^2$ to $\mathbf{R}$,then $P$ is not a local extreme point
  of $f$.
\end{lemma}
\begin{proof}
  The proof is left to the reader.
\end{proof}

First we prove the existence of the Fermat point of the triangle
$P_1P_2P_3$.Zuo Quanru and Lin Bo already used a
sophisticated version of this method in \cite{zuo}.

Let $P=(x_{P},y_{P})$,$P_1=(x_{P_{1}},y_{P_{1}}),P_2=(x_{P_{2}},y_{P_{2}}),P_3=(x_{P_{3}},y_{P_{3}})$.Let 
\begin{equation*}
  f(x,y)=|PP_1|+|PP_2|+|PP_3|=\sum_{i=1}^3\sqrt{(x_{P}-x_{P_{i}})^2+(y_{P}-y_{P_{i}})^2}.
\end{equation*}

\begin{theorem}[Existence of the Fermat point]\label{theorem:3}
  Any triangle $P_1P_2P_3$ has a Fermat point.
\end{theorem}
\begin{proof}
Draw a circle $O_1$ centering at $P_1$,whose radius $r$ is large
enough.Then $$D_{1}=\{|P-P_1|\leq r:P\in \mathbf{R}^2\}$$ is a bounded
closed disk.According to the extreme value theorem,$f$ must
attain a minimum on $D_1$.When $r$ is large,the minimum point of $f$ on
$D_1$ is the minimum point of $f$ on the whole plane $\mathbf{R}^2$.Thus the existence of the Fermat
point of the triangle $P_1P_2P_3$ is guaranteed.
\end{proof}

Now we prove the uniqueness of the Fermat point of the triangle
$P_1P_2P_3$,in the mean time,we find the exact
location of the Fermat point.Theorem \eqref{theorem:9}  and Theorem
\eqref{theorem:10} are our main theorems.

If $P_0=(x_{P_{0}},y_{P_{0}})$ is a minimum point of $f$,and $P_0\not\in \{P_1,P_2,P_3\}$,then
according to Fermat's theorem,we have
\begin{equation}\label{eq:1}
  \begin{cases}
          \displaystyle\frac{\partial f}{\partial x}(x_{P_{0}},y_{P_{0}})=\sum_{i=1}^3
  \frac{x_{P_{0}}-x_{P_{i}}}{\sqrt{(x_{P_{0}}-x_{P_{i}})^2+(y_{P_{0}}-y_{P_{i}})^2}}=0,\\
\displaystyle\frac{\partial f}{\partial
    y}(x_{P_{0}},y_{P_{0}})=\sum_{i=1}^3 \frac{y_{P_{0}}-y_{P_{i}}}{\sqrt{(x_{P_{0}}-x_{P_{i}})^2+(y_{P_{0}}-y_{P_{i}})^2}}=0.
  \end{cases}
\end{equation}
Let vectors
$$
\mathbf{L_{0}}=(x_{p_{0}}-x_{p_{1}},y_{p_{0}}-y_{p_{1}}),\mathbf{M_{0}}=(x_{p_{0}}-x_{p_{2}},y_{p_{0}}-y_{p_{2}}),\mathbf{N_{0}}=(x_{p_{0}}-x_{p_{3}},y_{p_{0}}-y_{p_{3}}).
$$
Then the simultaneous equations \eqref{eq:1} is equivalent to
\begin{equation}
  \label{eq:2}
\mathbf{\frac{L_{0}}{|L_{0}|}+\frac{M_{0}}{|M_{0}|}+\frac{N_{0}}{|N_{0}|}}=\mathbf{0}.
\end{equation}
When $P_0\in \{P_1,P_2,P_3\}$,equation \eqref{eq:2} is not
defined,because in this case,one of  $|\mathbf{L_{0}}|$,
$|\mathbf{M_{0}}|$,  $|\mathbf{N_0}|$ is $0$.

Notice that 
$\frac{\mathbf{L_0}}{|\mathbf{L_0}|},\frac{\mathbf{M_0}}{|\mathbf{M_0}|},\frac{\mathbf{N_0}}{|\mathbf{N_0}|}$
are unit vectors.According to lemma 5,it is easy to verify that equation
\eqref{eq:2} holds if and only if the point $P_0$ is in the interior of the
triangle region,and the radian measure of the angle between any
two of the unit vectors is $\frac{2\pi}{3}$.

If there is no point satisfying equation \eqref{eq:2},then there is no
minimum point of $f$ on $\mathbf{R}^2\backslash\{P_1,P_2,P_3\}$,which means that there is no Fermat point of the triangle except
points $P_1,P_2,P_3$.But according to the existence of the Fermat
point(Theorem \eqref{theorem:3}),we
know that the Fermat point of the triangle do exist,so the Fermat point of the triangle must be on the vertex of the triangle
$P_1P_2P_3$ whose corresponding interior angle is the
largest among the three interior angles.Combine the analysis in this
paragraph and in last paragraph with  lemma
\eqref{lemma:4},we have

\begin{theorem}\label{theorem:9}
When the radian measure of an interior angle of the triangle $P_1P_2P_3$
is equal or larger than $\frac{2\pi}{3}$,then the Fermat point must be on the vertex of
the triangle whose corresponding interior angle is the largest among
the three interior angles,and the Fermat point is unique.
\end{theorem}

Next we prove
\begin{theorem}\label{theorem:10}
When the radian  measure of all the interior angles of the triangle
$P_1P_2P_3$ are less than $\frac{2\pi}{3}$,then the Fermat point  must
be  in the interior of the triangle and is unique,denoted by $P_0$.And
$\mbox{rad}\angle P_1P_{0}P_2=\mbox{rad}\angle P_2P_{0}P_3=\mbox{rad}\angle P_3P_{0}P_1=\frac{2\pi}{3}$.
\end{theorem}
\begin{proof}
According to lemma \eqref{theorem:1.22},there exists a unique point
$P_{0}$ in the interior of the triangle  satisfying the condition
$\mbox{rad}\angle P_1P_{0}P_2=\mbox{rad}\angle P_2P_{0}P_3=\mbox{rad}\angle
P_3P_{0}P_1=\frac{2\pi}{3}$.

And the  point $P_0$ is the unique point satisfying equation
\eqref{eq:2},which means that $P_0$ is the only possible minimum point of $f$ except
$P_1,P_2,P_3$.So if we managed to prove that  none of the points $P_1,P_2,P_3$ is
the minimum point of $f$,then we managed to prove that $P_0$ is the unique
Fermat point of $f$.Now we do this job.

Draw a line $l$ passing through the point $P_1$.Now we prove that for any sequence of 
points $Q_1=(x_1,y_1),Q_2=(x_2,y_2),\cdots,Q_n=(x_n,y_n),\cdots$ on
$l$ such that $\lim_{n\to\infty}|Q_nP_1|=0$($P_1$ is not in the sequence),we have
$\lim_{n\to\infty}f'(Q_n)$ exists and is a nonzero linear map from
$\mathbf{R}^2$ to $\mathbf{R}$,then according to lemma 7,we can prove
that $P_1$ is not
a minimum point of $f$.

Let
$\mathbf{L_{n}}=(x_{n}-x_{P_{1}},y_{n}-y_{P_{1}}),\mathbf{M_{n}}=(x_{n}-x_{P_{2}},y_{n}-y_{P_{2}}),\mathbf{N_{n}}=(x_{n}-x_{P_{3}},y_{n}-y_{P_{3}})$.Then
$$
\frac{\mathbf{L_{n}}}{\mathbf{|L_{n}|}}+\frac{\mathbf{M_{n}}}{|\mathbf{M_{n}}|}+\frac{\mathbf{N_{n}}}{\mathbf{|N_{n}|}}=(\frac{\partial f}{\partial x}(x_n,y_n),\frac{\partial f}{\partial y}(x_n,y_n)).
$$
And we have
$$
\lim_{n\to\infty}\left|\frac{\mathbf{M_{n}}}{|\mathbf{M_{n}}|}+\frac{\mathbf{N_{n}}}{|\mathbf{N_{n}}|}\right|>1,
$$
this is because $\mbox{rad}\angle P_2P_1P_3<\frac{2\pi}{3}$.So as $n$
goes to
infinity,$\frac{\mathbf{L_{n}}}{\mathbf{|L_{n}|}}+\frac{\mathbf{M_{n}}}{|\mathbf{M_{n}}|}+\frac{\mathbf{N_{n}}}{\mathbf{|N_{n}|}}$
tends to a nonzero vector.So as $n$ goes to infinity,$f'(Q_n)$ tends
to a nonzero linear map from $\mathbf{R}^2$ to $\mathbf{R}$.Done.

So $P_1$ is not a minimum point of $f$.Similarly,$P_2,P_3$ are not
minimum points of $f$.So $P_0$ is the unique minimum point of $f$.
\end{proof}


\begin{thebibliography}{1}
\bibitem{zuo}Zuo Quanru,Lin Bo.Fermat Points of finite Point Sets in
  Metric Spaces[J].Journal of Mathematics.(PRC),1997-03
\end{thebibliography}
\end{document}